\documentclass[12pt]{article}

\usepackage{amsmath,amssymb,amsthm}
\usepackage{xcolor,colortbl}
\usepackage{enumitem}
\usepackage{graphicx}
\usepackage{tikz}
\usepackage[left=1.2in, right=1.2in, top=1.2in, bottom=1.25in]{geometry}

\newtheorem{theorem}{Theorem}[section]
\newtheorem{conjecture}[theorem]{Conjecture}

\newtheorem{fact}[theorem]{Fact}
\newtheorem{problem}[theorem]{Problem}

\newtheorem{observation}[theorem]{Observation}
\newtheorem{definition}[theorem]{Definition}

\newcommand{\ep}{\varepsilon}
\newcommand{\floor}[1]{\lfloor#1\rfloor}
\newcommand{\ceiling}[1]{\lceil#1\rceil}
\newcommand{\mc}{\mathrm{mc}}
\newcommand{\tbf}[1]{\textbf{#1}}

\usepackage{enumitem}
\setlist{itemsep=0pt, topsep=2pt}

\title{A note about monochromatic components in graphs of large minimum degree}
\author{Louis DeBiasio$^{1}$, Robert A. Krueger$^{2}$}
\date{\today}

\begin{document}

\maketitle
\noindent\footnotetext[1]{Department of Mathematics, Miami University {\tt debiasld@miamioh.edu}.}
\noindent\footnotetext[2]{Department of Mathematics, University of Illinois, Urbana-Champaign {\tt rak5@illinois.edu}.}

\begin{abstract}
For all positive integers $r\geq 3$ and $n$ such that $r^2-r$ divides $n$ and an affine plane of order $r$ exists, we construct an $r$-edge colored graph with minimum degree $(1-\frac{r-2}{r^2-r})n-2$ such that the largest monochromatic component has order less than $\frac{n}{r-1}$.  This generalizes an example of Guggiari and Scott and, independently, Rahimi for $r=3$ and thus disproves a conjecture of Gy\'arf\'as and S\'ark\"ozy for all integers $r\geq 3$ such that an affine plane of order $r$ exists.
\end{abstract}

\section{Introduction}

An \emph{affine plane of order $q$} is a $q$-uniform hypergraph on $q^2$ vertices (called points), with $q(q+1)$ edges (called lines) such that each pair of vertices is contained in exactly one edge.  It is well known that an affine plane of order $q$ exists whenever $q$ is a prime power (and it is unknown whether there exists an affine plane of non-prime power order).  Given an affine plane $\mathcal{G}$ of order $q$, there exists a $q+1$-coloring of the edges of $\mathcal{G}$ such that every color class (called a parallel class) consists of a collection of $q$ vertex disjoint edges of order $q$, every vertex is contained in exactly one edge of each color, and the union of the $q+1$ edges incident to a given vertex is all of $V(\mathcal{G})$.

Let $H=(\{x_1, \dots, x_t\},E)$ be a hypergraph which has a proper edge coloring with $r$ colors (that is, every color class induces a matching).  Let $\alpha = (\alpha_1, \dots, \alpha_t)\in \mathbb{R}^t$ be such that $\sum_{i=1}^t\alpha_i=1$ and $\alpha_i>0$ for all $i\in [t]$.  For a positive integer $n$, let $G$ be a graph on $n$ vertices obtained by replacing each $x_i\in V(H)$ with a set $X_i$ of order $\ceiling{\alpha_i n}$ or $\floor{\alpha_i n}$; for all $u\in X_i$, $v\in X_j$, let $uv$ be an edge of $G$ if and only if there exists $e\in E$ such that $\{x_i, x_j\}\subseteq e$, and color $uv$ using the color which appears on $e$ (if there are multiple such edges, choose a color arbitrarily from one such edge).  We call $G$ an $\alpha$-\emph{weighted blow-up} of $H$, and if $\alpha_i=\frac{1}{t}$ for all $i\in [t]$, we call $G$ a \emph{uniform blow-up} of $H$.

Given a graph $G$ and a positive integer $r$, let $\mc_r(G)$ be the largest integer $m$ such that in every $r$-edge-coloring of $G$, there exists a monochromatic component (i.e.\ a maximal connected subgraph) of order at least $m$.  For the rest of the paper, when we speak of an $r$-coloring of $G$, we mean an $r$-coloring of the edges of $G$.  

Gy\'arf\'as \cite{Gy} proved $$\mc_r(K_n)\geq \frac{n}{r-1}$$ and this is best possible when $(r-1)^2$ divides $n$ and an affine plane of order $r-1$ exists.  To see this, let $K_n$ be a uniform blow-up of the affine plane of order $r-1$.  Since every pair of distinct points from the affine plane is contained in exactly one edge the $r$-coloring of $K_n$ is well defined, and since each line of the affine plane has order $r-1$ and there are $(r-1)^2$ points, the size of the largest monochromatic component in $K_n$ is $(r-1)\frac{n}{(r-1)^2}=\frac{n}{r-1}$.

Gy\'arf\'as and S\'ark\"ozy \cite{GyS} raised the following interesting question: for a graph $G$ on $n$ vertices, how large does the minimum degree of $G$ need to be so that $\mc_r(G)\geq\frac{n}{r-1}$?  As noted in \cite{GyS0}, the answer is $n-1$ for $r=2$ because there is a 2-coloring of any non-complete graph on $n$ vertices such that the largest monochromatic component has order at most $n-1$.  So it was perhaps surprising that for all $r\geq 3$, they showed there exists $\ep_r>0$ such that if $G$ is a graph on $n$ vertices with $n$ sufficiently large and $\delta(G)\geq(1-\ep_r)n$, then $\mc_r(G)\geq \frac{n}{r-1}$.  The bounds on $\ep_r$ given in \cite{GyS} were later improved in \cite{DKS} as follows: for $r=3$, $\delta(G)\geq 7n/8$ suffices and for $r\geq 4$, $\delta(G)\geq (1-\frac{1}{3072(r-1)^5})n$ suffices.

Gy\'arf\'as and S\'ark\"ozy \cite{GyS} also gave the following natural construction whenever an affine plane of order $r$ exists and $r^2$ divides $n$.  Repeat the construction given above, but instead of an affine plane of order $r-1$, take a uniform blow-up of an affine plane of order $r$ with one parallel class removed.  This gives an $r$-colored graph on $n$ vertices with minimum degree $(1-\frac{r-1}{r^2})n-1$ where the largest monochromatic component has order $\frac{n}{r}<\frac{n}{r-1}$.  They conjectured that the bound arising from this construction is tight.

\begin{conjecture}[Gy\'arf\'as, S\'ark\"ozy \cite{GyS}]\label{exGS}
Let $n$ and $r\geq 3$ be positive integers.  If $G$ is a graph on $n$ vertices such that $\delta(G) \geq (1-\frac{r-1}{r^2})n$, then $\mc_r(G)\geq \frac{n}{r-1}$.
\end{conjecture}

Recently, Guggiari and Scott, and independently Rahimi, disproved this conjecture for $r=3$. The combination of their results gives the best possible minimum degree condition.

\begin{theorem}[Guggiari, Scott \cite{GS}, Rahimi \cite{R}]\label{GuggiariScott}
Let $G$ be a graph on $n$ vertices.  If $\delta(G) \geq \frac{5}{6} n-1$, then $\mc_3(G)\geq \frac{n}{2}$. Moreover, for every $n$, there exists a graph $G$ on $n$ vertices with $\delta(G) = \ceiling{\frac{5}{6} n} - 2$ such that $\mc_3(G)<\frac{n}{2}$.
\end{theorem}

Note that the $3$-colorings of graphs with $\delta(G) = \ceiling{\frac{5}{6}n} -2$ given by Guggiari and Scott and Rahimi have largest monochromatic components of order just under $\frac{n}{2}$. This is in contrast to the example of Gy\'arf\'as and S\'ark\"ozy above, where the largest monochromatic components have order $\frac{n}{3}$.

The purpose of this note is to generalize the lower bound construction of Guggiari and Scott and Rahimi which disproves Conjecture \ref{exGS} whenever an affine plane of order $r$ exists.

\begin{theorem}\label{main_example}
Let $n$ and $r$ be integers such that $r\geq 3$ and $n\geq r(r-1)((r-1)(r-2)+1)$.  If $(r^2-r) \mid n$ and an affine plane of order $r$ exists, then there exists a graph $G$ on $n$ vertices with $$\delta(G) = \left(1-\frac{r-2}{r^2-r}\right)n-2=\left(1-\frac{r-1}{r^2}+\frac{1}{r^2(r-1)}\right)n-2$$ such that $\mc_r(G)< \frac{n}{r-1}$.
\end{theorem}

%Let $f_r(n)$ be the smallest $d$ such that for every graph $G$ on $n$ vertices with minimum degree $\delta(G) \geq d$ we have $\mc_r(G)=\mc_r(K_n)\geq \frac{n}{r-1}$.  Let $\ep_r = 1-\limsup_{n\to\infty} f_r(n)/n$ be (roughly) proportion of each vertices' neighbors in $K_n$ we are allowed to remove. \bkc{Can we prove that the limit exists? If we can show that $f_r(n)$ is monotonic, then the limit exists. I can do this for $r=3$ but not $r \geq 4$.} 

% (More generally, one can ask what is the largest monochromatic component guaranteed in an $r$-coloring of $G$ with $\delta(G) \geq d$ for every $d \in [n]$. Indeed, Gy\'arf\'as and S\'ark\"ozy did this for $r=2$ \cite{GyS}. This may be feasible for $r=3$, but for larger $r$, the difficulty grows very quickly.)

The construction is based on a blow-up of the following hypergraph $\mathcal{H}_r$ which is derived from an affine plane of order $r$.

\begin{definition}[$\mathcal{H}_r$]\label{H_r}
Let $r\geq 3$ such that an affine plane of order $r$ exists.  Let $\mathcal{G}_r=(V, \mathcal{L})$ be an affine plane of order $r$.  Let $\{L_1, \dots, L_{r+1}\}$ be the partition of $\mathcal{L}$ into parallel classes.  Label the vertices of $\mathcal{G}_r$ as $v_{i,j}$ with $i,j \in [r]$ so that $L_1=\{\{v_{i,1}, \dots, v_{i,r}\}: i\in [r]\}$ and $L_{r+1}=\{\{v_{1,i}, \dots, v_{r,i}\}: i\in [r]\}$ 
(in Figure \ref{H_rfigure}, $L_1$ is represented by the rows and $L_{r+1}$ by the columns).
Let $S=\{v_{r,i} : i\in [r-1]\} \cup \{v_{r-1,r}\}$. 

Let $\mathcal{H}_r$ be the hypergraph obtained from $\mathcal{G}_r$ by deleting the lines from $L_{r+1}$ and the vertices of $S$ from each of the remaining edges; i.e.\ let $\mathcal{H}_r=(V\setminus S, E)$ where $E=\{e\setminus S: e\in \mathcal{L}\}$.
\end{definition}

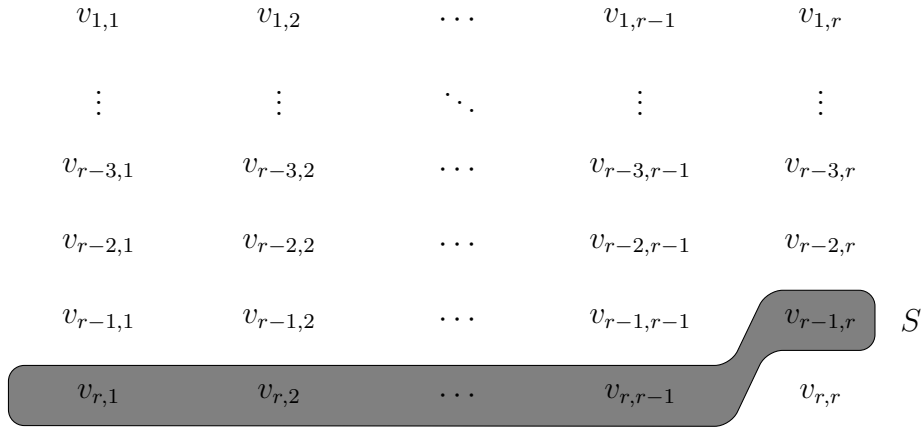
\begin{figure}[ht]
\centering
\begin{tikzpicture}
\def\x{2.4};
\def\y{1};
\def\round{6pt}

\draw[rounded corners=\round,fill=gray] (-0.5*\x,-0.4*\y) -- ++(0,0.8*\y) -- ++(4.0*\x,0) -- ++(0.2*\x,1*\y) -- ++(0.6*\x,0) -- ++(0,-0.8*\y) -- ++(-0.6*\x,-0.0*\y) -- ++(-0.2*\x,-1.0\y) -- cycle;
\draw (4.5*\x,1*\y) node{$S$};
%\draw[rounded corners=\round] (-0.5*\x,0.6*\y) -- ++(0,0.8*\y) -- ++(3.9*\x,0) -- ++(0.2*\x,1*\y) -- ++(0.9*\x,0) -- ++(0,-0.8*\y) -- ++(-0.9*\x,-0.0*\y) -- ++(-0.2*\x,-1.0\y) -- cycle;
%\draw (4.7*\x,2*\y) node{$A$};

\draw (0*\x,0*\y) node{$v_{r,1}$};
\draw (1*\x,0*\y) node{$v_{r,2}$};
\draw (2*\x,0*\y) node{$\cdots$};
\draw (3*\x,0*\y) node{$v_{r,r-1}$};
\draw (4*\x,0*\y) node{$v_{r,r}$};

\draw (0*\x,1*\y) node{$v_{r-1,1}$};
\draw (1*\x,1*\y) node{$v_{r-1,2}$};
\draw (2*\x,1*\y) node{$\cdots$};
\draw (3*\x,1*\y) node{$v_{r-1,r-1}$};
\draw (4*\x,1*\y) node{$v_{r-1,r}$};

\draw (0*\x,2*\y) node{$v_{r-2,1}$};
\draw (1*\x,2*\y) node{$v_{r-2,2}$};
\draw (2*\x,2*\y) node{$\cdots$};
\draw (3*\x,2*\y) node{$v_{r-2,r-1}$};
\draw (4*\x,2*\y) node{$v_{r-2,r}$};

\draw (0*\x,3*\y) node{$v_{r-3,1}$};
\draw (1*\x,3*\y) node{$v_{r-3,2}$};
\draw (2*\x,3*\y) node{$\cdots$};
\draw (3*\x,3*\y) node{$v_{r-3,r-1}$};
\draw (4*\x,3*\y) node{$v_{r-3,r}$};

\draw (0*\x,4*\y) node{$\vdots$};
\draw (1*\x,4*\y) node{$\vdots$};
\draw (2*\x,4*\y) node{$\ddots$};
\draw (3*\x,4*\y) node{$\vdots$};
\draw (4*\x,4*\y) node{$\vdots$};

\draw (0*\x,5*\y) node{$v_{1,1}$};
\draw (1*\x,5*\y) node{$v_{1,2}$};
\draw (2*\x,5*\y) node{$\cdots$};
\draw (3*\x,5*\y) node{$v_{1,{r-1}}$};
\draw (4*\x,5*\y) node{$v_{1,r}$};

% \draw[rounded corners=\round] (-0.5*\x, 0.6*\y) rectangle (3.5*\x, 1.4*\y) {};
% \draw (-0.7*\x,\y) node{$A$};
% \draw[rounded corners=\round] (3.8*\x, 1.5*\y) rectangle (4.2*\x, 5.5*\y) {};
% \draw (4.4*\x,3.5*\y) node{$B$};

\end{tikzpicture}
\caption{The hypergraph $\mathcal{H}_r$.}\label{H_rfigure}
\end{figure}

Given a hypergraph $H=(V,E)$, the \emph{rank} of $H$, denoted $r(H)$, is $\max\{|e|: e\in E\}$ and the \emph{proportional rank} of $H$ is $\frac{r(H)}{|V|}$.  The \emph{edge chromatic number} of $H$ is the minimum number of colors needed to color the edges of $H$ so that each color class forms a matching. Given a vertex $v\in V$, let $N[v]=\{u: \exists e\in E, \{u,v\}\subseteq e\}$; in other words, $N[v]$ is the set of all vertices (including $v$) which are contained in an edge with $v$. Let $\delta^*(H)=\min\{|N[v]|: v\in V\}$.

Note the following properties of $\mathcal{H}_r=(V,E)$:
\begin{enumerate}[label={(P\arabic*)}]
\item \label{itm:P1} the edge chromatic number of $\mathcal{H}_r$ is $r$,
\item \label{itm:P2} the proportional rank of $\mathcal{H}_r$ is $\frac{r}{r^2-r}=\frac{1}{r-1}$,
\item \label{itm:P3} $\delta^*(\mathcal{H}_r)=|V|-(r-2)=(1-\frac{r-2}{r^2-r})|V|$.
\end{enumerate}

Roughly speaking, we prove Theorem \ref{main_example} by taking a uniform blow-up of $\mathcal{H}_r$ (which has monochromatic components of order $\frac{n}{r-1}$) and then slightly  ``perturbing" the sizes of the blown-up sets so that all the monochromatic components have order less than $\frac{n}{r-1}$.  This raises the more general question of when such a perturbation is possible, which we address in Section \ref{sec:perturb}.

As is elaborated in Section \ref{sec:rough}, the choice of vertices $S$ to delete in the definition of $\mathcal{H}_r$ is to ensure that a uniform blow-up of $\mathcal{H}_r$ is ``perturbable.'' As an example of a hypergraph which is not ``perturbable,'' let $\mathcal{H}_3'=(V,E)$ be obtained from an affine plane of order 3, by deleting one parallel class and deleting the vertices from one of the remaining edges (say $v_{3,1}, v_{3,2}, v_{3,3}$ as in Figure \ref{H_rfigure}).  The edge chromatic number of $\mathcal{H}_3'$ is 3, the proportional rank is $1/2$, and $\delta^*(\mathcal{H}_3')=5=(1-\frac{1}{6})|V|$.  By taking a uniform blow-up of $\mathcal{H}_3'$ we obtain a 3-colored graph $G$ with $\delta(G)= \frac{5n}{6}-1$, and every monochromatic component has order at most $n/2$.  However, no matter how we change the sizes of the blown-up sets, one of the monochromatic components will have order at least $n/2$. 

\section{Perturbable hypergraphs}\label{sec:perturb}

It is possible to skip directly to Section \ref{skipto} to see the proof of Theorem \ref{main_example}; however, to understand where the construction comes from we need to take a slight detour.

The \emph{standard simplex} of $\mathbb{R}^n$ is the set of vectors $(w_1, \dots, w_n)$ such that $w_i\geq 0$ for all $i\in [n]$ and $\sum_{i=1}^nw_i=1$.  A \emph{weight assignment} on a hypergraph $H=(V,E)$ where $V=\{v_1, \dots, v_n\}$ is a function $w:V\to \mathbb{R}$ such that $(w(v_1), \dots, w(v_n))$ is in the standard simplex of $\mathbb{R}^n$. For all $S\subseteq V$ let the weight of $S$, denoted $w(S)$, be $\sum_{v\in S}w(v)$.   We say that $w:V\to \mathbb{R}$ given by $w(v)=\frac{1}{|V|}$ for all $v\in V$ is the \emph{uniform} weight assignment.

\begin{definition}[Perturbation, perturbable]
A \emph{perturbation} on a hypergraph $H = (V,E)$ is a function $p:V\to\mathbb{R}$ such that $\sum_{i=1}^n p(v_i) = 0$ and for all $e \in E$, $p(e) = \sum_{v\in e} p(v) < 0$. We say $H$ is \emph{perturbable} if a perturbation on $H$ exists.
\end{definition}

Observe that if $w$ is a positive weight assignment on $H$ (meaning $w(v) > 0$ for all $v \in V$) and $p$ is a perturbation on $H$, then $w+\ep p$ is also a weight assignment on $H$ for sufficiently small $\ep>0$ (say $\ep < \min \left\{\left|\frac{w(v)}{p(v)}\right|: v\in V, p(v)\neq 0\right\}$). Since $p(e) < 0$ for every $e \in E$, the weight assignment $w+\ep p$ is strictly smaller than $w$ on every edge of $H$. Thus, if a perturbation on $H$ exists, then we can alter any weight assignment by at most $\ep$ at each vertex (for $\ep$ sufficiently small) and \emph{strictly} decrease the weights on the edges.

Theorem \ref{perturbable} gives an equivalent condition for the existence of a perturbation, but first we must recall the following definitions. Given a hypergraph $H=(V,E)$, a \emph{fractional matching} is a function $m:E\to [0,1]$ such that for all $v\in V$, $\sum_{e\ni v}m(e)\leq 1$, and a \emph{fractional vertex cover} is a function $t:V\to [0,1]$ such that for all $e\in E$, $\sum_{v\in e}t(v)\geq 1$. A fractional matching is called \emph{perfect} if we have equality for all $v\in V$. We let 
\[
\nu^*(H)=\max\left\{\sum_{e\in E}m(e): m \text{ is a fractional matching on } H\right\} 
\]
and 
\[
\tau^*(H)=\min\left\{\sum_{v\in V}t(v): t \text{ is a fractional vertex cover on } H\right\}.  
\]
It is well known consequence of the duality theorem in linear programming that $\tau^*(H)=\nu^*(H)$ for all hypergraphs.  When it is clear from context we just write $\tau^*$ and $\nu^*$ for $\tau^*(H)$ and $\nu^*(H)$ respectively. Note that if $H$ is $k$-uniform and has a perfect fractional matching, then $\nu^*(H) = \frac{n}{k}$.

\begin{theorem}\label{perturbable}
Let $H=(V,E)$ be a hypergraph.  $H$ is perturbable if and only if $H$ does not have a perfect fractional matching.
\end{theorem}

\begin{proof}
Let $n:=|V|$ and $e:=|E|$ and let $A$ be the $n$-by-$e$ incidence matrix of $H$ (with rows indexed by vertices and columns by edges).  Let $\mathbf{1}$ be the $n$-dimensional vector of all 1's.  Note that in this language, a perfect fractional matching $m$ is a solution to the system $Am = \mathbf{1}$, $m \geq 0$, and a perturbation $p$ is a solution to the system $A^T p < 0$, $\mathbf{1}^T p = 0$.

Recall Farkas' Lemma (see \cite{LP}), which states for a $n$-by-$e$ matrix $A$ and $n$-dimensional vector $b$, there is no $m\geq 0$ such that $Am = b$ if and only if there exists $w$ such that $A^T w \leq 0$ and $b^T w > 0$.

We claim that for the given $A$, the solvability of $A^T w \leq 0$, $\mathbf{1}^T w > 0$ is equivalent to the solvability of $A^T p < 0$, $\mathbf{1}^T p = 0$. So by Farkas' lemma, the result will follow by establishing this claim.

%Additionally, in the system $A^T y \leq 0$, $b^T y > 0$ we can be interpret $y$ as vertex weights, which are nonpositive on every edge but whose total is positive. 

First suppose there exists $w$ such that $A^T w \leq 0$ and $\mathbf{1}^T w > 0$. Letting $p = w - \frac{\mathbf{1}^Tw}{n}\mathbf{1}$, we have $\mathbf{1}^T p = 0$ and $p < w$.  Since $A$ has only nonnegative entries, and it has at least one positive entry in each column, we have $A^T p < A^T w \leq 0$, so $p$ is a perturbation for $H$.

For the other direction, suppose there exists $p$ such that $A^T p < 0$ and $\mathbf{1}^T p = 0$. Let $\alpha>0$ be the absolute value of the largest entry of $A^T p$ (smallest in absolute value), and let $w = p + \frac{\alpha}{n}\mathbf{1}$. Then $\mathbf{1}^T w=\alpha > 0$ and the largest entry of $A^T w=A^Tp+\frac{\alpha}{n} A^T \mathbf{1}$ is $-\alpha + \frac{\alpha}{n} n \leq 0$, since the largest entry of $A^T \mathbf{1}$ is the rank of $H$ which is at most $n$.
\end{proof}

\section{When an affine plane of order $r$ exists}\label{sec:main}

Given a hypergraph $H=(V,E)$ and a weight assignment $w$, the \emph{top-level} of $H$, denoted $\overline{H}$, is the hypergraph $(V,E')$ where $E'\subseteq E$ is the set of edges of maximum weight.

\subsection{Rough construction}\label{sec:rough}

Let $\mathcal{H}_r=(V,E)$ be the hypergraph from Definition \ref{H_r}.  We first show that under the uniform weight assignment, the top-level of $\mathcal{H}_r=(V,E)$ is perturbable. This together with properties \ref{itm:P1}, \ref{itm:P2}, and \ref{itm:P3}, imply that for all sufficiently large $n$ we can use $\mathcal{H}_r$ to define a graph $G$ on $n$ vertices with $\delta(G)\geq (1-\frac{r-2}{r^2-r})n-O(1)$ and $\mc_r(G)<\frac{n}{r-1}$.

\begin{observation}\label{rough}
Let $r\geq 3$ be an integer such that an affine plane of order $r$ exists.  Under the uniform weight assignment, the top-level of $\mathcal{H}_r=(V,E)$ is perturbable.
\end{observation}

\begin{proof}
%Let $\mathcal{H}_r$ be the hypergraph given in Definition \ref{H_r}. 
Given the uniform weight assignment on $\mathcal{H}_r$, let $\overline{\mathcal{H}_r}=(V,E')$ be the top-level of $\mathcal{H}_r$ (which is just the edges of order $r$ in this case).  
% (in fact, there are edges of order $r$, edges of order $r-1$, and one edge of order $1$).  Importantly every (remnant of a) parallel class in $\mathcal{H}$ contains $r$ edges.  Also important is that every edge of order $r$ in $\mathcal{H}$ intersects $A$ in at most one vertex.  (The only edges of order $r$ either come from $L_1$ or contain $v_{r,r}$).
Let $t:V\to \mathbb{R}$ given by $t(v)=\frac{1}{r-1}$ if $v\in \{v_{r,r}\}\cup \{v_{i,j}: i\in [r-2], j\in [r-1]\}$ and $t(v)=0$ otherwise (see Figure \ref{roughfigure}).  We first claim that $t$ is a fractional vertex cover of $\overline{\mathcal{H}_r}$.  Indeed, every edge of $E'$ either comes from $L_1$ or contains $v_{r,r}$, and consequently intersects $\{v_{r-1,i}: i\in [r-1]\}\cup \{v_{i,r}: i\in [r-2]\}$ in at most one vertex.  So we have
\[
\nu^*=\tau^*\leq \sum_{v\in V}t(v)=(r-1)(r-2)\frac{1}{r-1}+\frac{1}{r-1}=r-2+\frac{1}{r-1}<r-1=\frac{|V|}{r},
\]  
and thus $\overline{\mathcal{H}_r}$ does not have a perfect fractional matching. By Theorem \ref{perturbable}, $\overline{\mathcal{H}_r}$ is perturbable.
\end{proof}

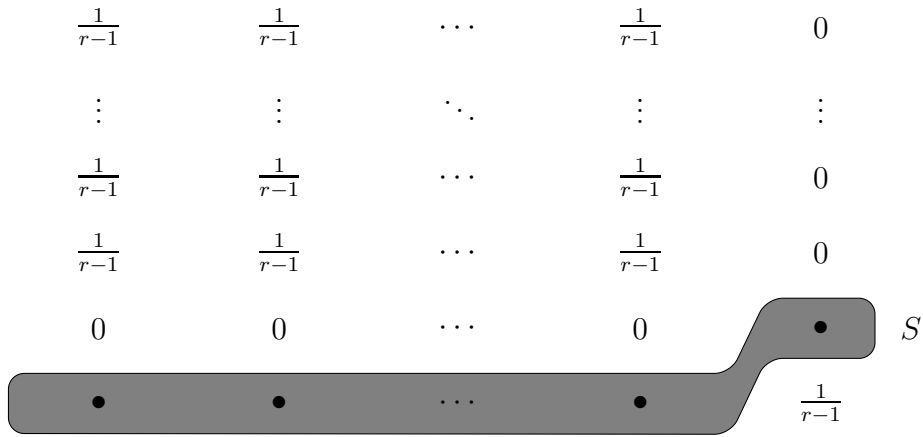
\begin{figure}[ht]
\centering
\begin{tikzpicture}
\def\x{2.4};
\def\y{1};
\def\round{6pt}

\draw[rounded corners=\round,fill=gray] (-0.5*\x,-0.4*\y) -- ++(0,0.8*\y) -- ++(4.0*\x,0) -- ++(0.2*\x,1*\y) -- ++(0.6*\x,0) -- ++(0,-0.8*\y) -- ++(-0.6*\x,-0.0*\y) -- ++(-0.2*\x,-1.0\y) -- cycle;
\draw (4.5*\x,1*\y) node{$S$};
%\draw[rounded corners=\round] (-0.5*\x,0.6*\y) -- ++(0,0.8*\y) -- ++(3.9*\x,0) -- ++(0.2*\x,1*\y) -- ++(0.9*\x,0) -- ++(0,-0.8*\y) -- ++(-0.9*\x,-0.0*\y) -- ++(-0.2*\x,-1.0\y) -- cycle;
%\draw (4.7*\x,2*\y) node{$A$};

\draw (0*\x,0*\y) node{\textbullet};
\draw (1*\x,0*\y) node{\textbullet};
\draw (2*\x,0*\y) node{$\cdots$};
\draw (3*\x,0*\y) node{\textbullet};
\draw (4*\x,0*\y) node{$\frac{1}{r-1}$};

\draw (0*\x,1*\y) node{0};
\draw (1*\x,1*\y) node{0};
\draw (2*\x,1*\y) node{$\cdots$};
\draw (3*\x,1*\y) node{0};
\draw (4*\x,1*\y) node{\textbullet};

\draw (0*\x,2*\y) node{$\frac{1}{r-1}$};
\draw (1*\x,2*\y) node{$\frac{1}{r-1}$};
\draw (2*\x,2*\y) node{$\cdots$};
\draw (3*\x,2*\y) node{$\frac{1}{r-1}$};
\draw (4*\x,2*\y) node{0};

\draw (0*\x,3*\y) node{$\frac{1}{r-1}$};
\draw (1*\x,3*\y) node{$\frac{1}{r-1}$};
\draw (2*\x,3*\y) node{$\cdots$};
\draw (3*\x,3*\y) node{$\frac{1}{r-1}$};
\draw (4*\x,3*\y) node{0};

\draw (0*\x,4*\y) node{$\vdots$};
\draw (1*\x,4*\y) node{$\vdots$};
\draw (2*\x,4*\y) node{$\ddots$};
\draw (3*\x,4*\y) node{$\vdots$};
\draw (4*\x,4*\y) node{$\vdots$};

\draw (0*\x,5*\y) node{$\frac{1}{r-1}$};
\draw (1*\x,5*\y) node{$\frac{1}{r-1}$};
\draw (2*\x,5*\y) node{$\cdots$};
\draw (3*\x,5*\y) node{$\frac{1}{r-1}$};
\draw (4*\x,5*\y) node{0};

% \draw[rounded corners=\round] (-0.5*\x, 0.6*\y) rectangle (3.5*\x, 1.4*\y) {};
% \draw (-0.7*\x,\y) node{$A$};
% \draw[rounded corners=\round] (3.8*\x, 1.5*\y) rectangle (4.2*\x, 5.5*\y) {};
% \draw (4.4*\x,3.5*\y) node{$B$};

\end{tikzpicture}
\caption{The fractional vertex cover of $\overline{\mathcal{H}_r}$.}\label{roughfigure}
\end{figure}

One may wonder if other choices of $S$ in the definition of $\mathcal{H}_r$ would yield a perturbable hypergraph satisfying properties \ref{itm:P1}, \ref{itm:P2}, and \ref{itm:P3}. An exhaustive search shows that there are no other choices of $S$ (up to isomorphism) for $r = 3$ and $r = 4$, but there are other choices for say $r=5$.  While it would be interesting to characterize the possible choices of $S$, doing so would not improve the given construction.

\subsection{Fine tuning}\label{skipto}

\begin{theorem}\label{planeminus}
Let $n,r,c$ be integers such that $r\geq 3$, $c\geq 1$, and $n\geq r(r-1)((r-1)(r-2)+1)c$.  If $(r^2-r) \mid n$ and an affine plane of order $r$ exists, then there exists a graph $G$ on $n$ vertices with $\delta(G) = (1-\frac{r-2}{r^2-r})n-c-1$ such that $\mc_r(G)\leq \frac{n}{r-1}-c$.
\end{theorem}

Note that the main case of interest is when $c=1$, but phrasing the result in general as we do shows that by lowering the minimum degree further, one can further decrease the size of the largest monochromatic component.  Also note that our construction only addresses the case $(r^2-r) \mid n$ for simplicity.  
It is possible that in the case when $r^2-r$ does not divide $n$, by slightly modifying this construction (as was done in \cite{GS} for the case $r=3$), one can construct a graph $G$ with $\delta(G) =\ceiling{(1-\frac{r-2}{r^2-r})n}-2$ such that $\mc_r(G)\leq \ceiling{\frac{n}{r-1}}-1$.

\begin{proof}
Let $\mathcal{H}_r=(V,E)$ be the hypergraph from Definition \ref{H_r} and let $G$ be a uniform blow-up of $\mathcal{H}_r$ where $v\in V$ becomes $X_v$ in $G$ (with $|X_v|=\frac{n}{r^2-r}$).  

Let $A=\{v_{r-1, i} : i\in [r-1]\}\cup \{v_{r-2, r}\}$.  We now adjust the size of each $X_v$ as follows:
\[ |X_v| := \begin{cases}
\frac{n}{r^2-r}-c      ,& \text{ if } v \in V(\mathcal{H}_r)\setminus A ,\\
\frac{n}{r^2-r}+(r-2)c ,& \text{ if } v \in A .
\end{cases} \]

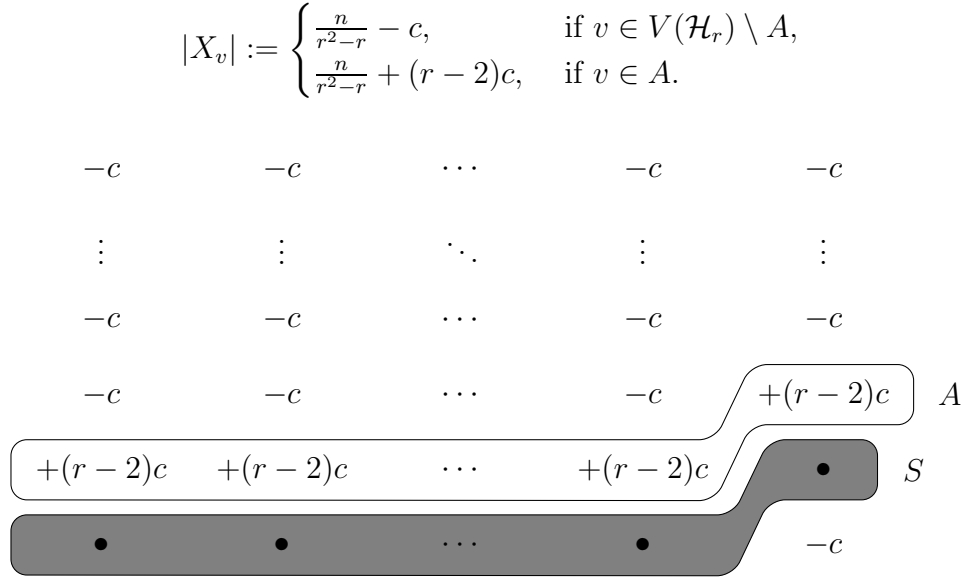
\begin{figure}[ht]
\centering
\begin{tikzpicture}
\def\x{2.4};
\def\y{1};
\def\round{6pt}

\draw[rounded corners=\round,fill=gray] (-0.5*\x,-0.4*\y) -- ++(0,0.8*\y) -- ++(4.0*\x,0) -- ++(0.2*\x,1*\y) -- ++(0.6*\x,0) -- ++(0,-0.8*\y) -- ++(-0.6*\x,-0.0*\y) -- ++(-0.2*\x,-1.0\y) -- cycle;
\draw (4.5*\x,1*\y) node{$S$};
\draw[rounded corners=\round] (-0.5*\x,0.6*\y) -- ++(0,0.8*\y) -- ++(3.9*\x,0) -- ++(0.2*\x,1*\y) -- ++(0.9*\x,0) -- ++(0,-0.8*\y) -- ++(-0.9*\x,-0.0*\y) -- ++(-0.2*\x,-1.0\y) -- cycle;
\draw (4.7*\x,2*\y) node{$A$};

\draw (0*\x,0*\y) node{\textbullet};
\draw (1*\x,0*\y) node{\textbullet};
\draw (2*\x,0*\y) node{$\cdots$};
\draw (3*\x,0*\y) node{\textbullet};
\draw (4*\x,0*\y) node{$-c$};

\draw (0*\x,1*\y) node{$+(r-2)c$};
\draw (1*\x,1*\y) node{$+(r-2)c$};
\draw (2*\x,1*\y) node{$\cdots$};
\draw (3*\x,1*\y) node{$+(r-2)c$};
\draw (4*\x,1*\y) node{\textbullet};

\draw (0*\x,2*\y) node{$-c$};
\draw (1*\x,2*\y) node{$-c$};
\draw (2*\x,2*\y) node{$\cdots$};
\draw (3*\x,2*\y) node{$-c$};
\draw (4*\x,2*\y) node{$+(r-2)c$};

\draw (0*\x,3*\y) node{$-c$};
\draw (1*\x,3*\y) node{$-c$};
\draw (2*\x,3*\y) node{$\cdots$};
\draw (3*\x,3*\y) node{$-c$};
\draw (4*\x,3*\y) node{$-c$};

\draw (0*\x,4*\y) node{$\vdots$};
\draw (1*\x,4*\y) node{$\vdots$};
\draw (2*\x,4*\y) node{$\ddots$};
\draw (3*\x,4*\y) node{$\vdots$};
\draw (4*\x,4*\y) node{$\vdots$};

\draw (0*\x,5*\y) node{$-c$};
\draw (1*\x,5*\y) node{$-c$};
\draw (2*\x,5*\y) node{$\cdots$};
\draw (3*\x,5*\y) node{$-c$};
\draw (4*\x,5*\y) node{$-c$};

% \draw[rounded corners=\round] (-0.5*\x, 0.6*\y) rectangle (3.5*\x, 1.4*\y) {};
% \draw (-0.7*\x,\y) node{$A$};
% \draw[rounded corners=\round] (3.8*\x, 1.5*\y) rectangle (4.2*\x, 5.5*\y) {};
% \draw (4.4*\x,3.5*\y) node{$B$};

\end{tikzpicture}
\caption{The adjustment of the sizes of the sets in a uniform blow-up of $\mathcal{H}_r$. Each number corresponds to a vertex of $\mathcal{H}_r$, with the rows corresponding to $L_1$ and the columns corresponding to $L_{r+1}$. The number at a vertex $v$ is the amount in which we adjusted the size of $X_v$, i.e.\ $|X_v| - \frac{n}{r^2-r}$.}
\label{planefigure}
\end{figure}

First note that $\sum_{v\in V}|X_v|=n$ (since each column sums to 0 in Figure \ref{planefigure}).

Now we check the minimum degree condition.  Let $v\in V(\mathcal{H}_r)$, let $\ell_v\in L_{r+1}$ such that $v\in \ell_v$, and let $u\in X_v$.  We have (see Figure \ref{planefigure} in which each vertex is adjacent to everything except the distinct members of its own column)
\begin{align*}
d(u)&=(n-1)-\sum_{w\in \ell_v\setminus \{v\}}|X_w|\\ 
&= \begin{cases}
(n-1)-(r-3)(\frac{n}{r^2-r}-c)-(\frac{n}{r^2-r}+(r-2)c)    ,& \text{ if } v \in V(\mathcal{H}_r)\setminus A ,\\
(n-1)-(r-2)(\frac{n}{r^2-r}-c) ,& \text{ if } v \in A,
\end{cases}
\end{align*}
and thus $\delta(G) = (n-1)-(r-2)\frac{n}{r^2-r}-c= (1-\frac{r-2}{r^2-r})n - c-1$.

Finally we check that every monochromatic component of $G$, which corresponds to an edge $\ell$ from $\mathcal{H}_r$, has order at most $\frac{n}{r-1}-c$. Since $c \leq \frac{n}{r(r-1)((r-1)(r-2)+1)}$, we have $(r-1)\left(\frac{n}{r^2-r}+(r-2)c\right)\leq \frac{n}{r-1}-c$ and thus we need only consider the edges $\ell$ of $\mathcal{H}_r$ of order $r$; that is, when $\ell \cap S = \emptyset$. Since every edge from $\mathcal{H}_r$ of order $r$ intersects $A$ in at most one vertex, the order of the largest monochromatic component in $G$ will be at most 
\[
(r-1)\left(\frac{n}{r^2-r}-c\right)+\frac{n}{r^2-r}+(r-2)c=\frac{n}{r-1}-c.
\]
\end{proof}

\section{When an affine plane of order $r$ does not exist}

It is known that there is no affine plane of order 6, so $r=6$ is the first case for which the construction of the previous section does not apply.  An example of a graph $G$ with $\delta(G)=\frac{5n}{7}-1$ such that $\mc_6(G)<\frac{n}{5}$ (in fact, $\mc_6(G)\leq \frac{n}{7}<\frac{n}{5}$) is a uniform blow-up of an affine plane of order 7 with two parallel classes removed.  

\begin{problem}
Construct an example of a graph $G$ with large minimum degree such that $\mc_6(G)<\frac{n}{5}$.  In particular, for some $\alpha>0$ and all $n$, construct a graph $G$ on $n$ vertices with $\delta(G)\geq (\frac{5}{7}+\alpha)n$ such that $\mc_6(G)<\frac{n}{5}$.  
\end{problem}

In light of Section \ref{sec:perturb}, it would suffice to construct a hypergraph $H=(V,E)$ with edge chromatic number 6, proportional rank at most $\frac{1}{5}$, $\delta^*(H)>\frac{5}{7}|V|$ such that if the proportional rank of $H$ is equal to $\frac{1}{5}$, then the top-level of $H$ with respect to the uniform weight assignment (the edges of maximum rank) has no perfect fractional matching.

In general, when an affine plane of order $r$ does not exist, trying to produce an example of an $r$-colored graph $G$ with large minimum degree for which $\mc_r(G)<\frac{n}{r-1}$ leads us back to the original problem for complete graphs.  The purpose of this section is mostly to collect what is known in one place and make a few observations.  These observations have consequences for the original problem for complete graphs and may be useful for extending our construction in the case when an affine plane of order $r$ does not exist.  

Recall that Gy\'arf\'as \cite{Gy} proved $\mc_r(K_n)\geq \frac{n}{r-1}$ and this is best possible when $(r-1)^2$ divides $n$ and an affine plane of order $r-1$ exists.  For all $r$ such that affine plane of order $r-1$ does not exist, the problem of determining $\mc_r(K_n)$ (even asymptotically) is still open.  The following result of F\"uredi shows that one can improve the lower bound on $\mc_r(K_n)$ when there is no affine plane of order $r-1$ (note that the upper bound comes from the construction mentioned in the introduction).

\begin{theorem}[F\"uredi \cite{F1}]\label{furedi}
Let $r\geq 3$ be an integer, let $q$ be the largest integer at most $r-1$ such that there exists an affine plane of order $q$, and let $n\geq q^2$ be an integer.  If an affine plane of order $r-1$ does not exist, then $\frac{n}{r-1-\frac{1}{r-1}}\leq \mc_r(K_n)\leq \ceiling{\frac{n}{q}}$.
\end{theorem}

Since an affine plane is a hypergraph in which every pair of distinct vertices is contained in exactly one edge and the edges of the hypergraph can be decomposed into perfect matchings (and thus has the smallest possible edge chromatic number), a natural place to look for examples which improve the upper bound (when an affine plane of order $r-1$ does not exist) are resolvable balanced incomplete block designs. 
%Of course the problem is more general than this (see \cite{F1} for instance), but this is a natural starting point.

A \emph{$(v,k,1)$-resolvable balanced incomplete block design}, a \emph{$(v,k,1)$-RBIBD} for short, is a $k$-uniform hypergraph $H$ on $v$ vertices such that each pair of vertices is contained in exactly one edge and the edges of $H$ can be decomposed into $\frac{\binom{v}{2}/\binom{k}{2}}{v/k}=\frac{v-1}{k-1}$ perfect matchings.  A necessary condition for the existence of a $(v,k,1)$-RBIBD is that $v\equiv k\bmod k(k-1)$.  Ray-Chaudhuri and Wilson \cite{RW} proved that for all $k\geq 3$ there exists a constant $C(k)$ such that if $v\geq  C(k)$ and $v\equiv k\bmod k(k-1)$, then a $(v,k,1)$-RBIBD exists.  Later Chang \cite{yC} proved that $C(k) = \exp(\exp(k^{12k^2}))$ suffices. 
%(this bound is a bit smaller when $k$ is even -- see Theorem 7.1 in \cite{yC}).   
%It is also known (see \cite{AGY}) that if $v\equiv k\bmod k(k-1)$ and $v$ and $k$ are both powers of the same prime, then a $(v,k,1)$-RBIBD exists. 
There are some other sporadic results for small $k$ (see \cite{AGY}), but in general, the existence of $(v,k,1)$-RBIBDs is open.

Note that an affine plane of order $k$ is a $(k^2,k,1)$-RBIBD and by the necessary condition above, $k^2$ is the smallest $v$ for which a non-trivial $(v,k,1)$-RBIBD exists.  Because of this, we parameterize $v$ in terms of $k$ and a non-negative integer $t$, and speak of $(k^2+tk(k-1),k,1)$-RBIBDs.

Given a hypergraph $H$, let $v(H)$ be the number of vertices in $H$ and recall that $r(H)$ is the rank of $H$.

\begin{fact}\label{RBIBD}
Let $k\geq 2$, $t\geq 0$, and $n\geq k^2+tk(k-1)$ be integers such that $k^2+tk(k-1)$ divides $n$.  If there exists a $(k^2+tk(k-1),k,1)$-RBIBD, then there is a $((t+1)k+1)$-coloring of $K_n$ such that every monochromatic component has order at most $\frac{n}{(t+1)k-t}$.  

In particular, when $t=0$ this means that if there exists an affine plane of order $k$, then there exists $(k+1)$-coloring of $K_n$ such that every monochromatic component has order at most $\frac{n}{k}$.
\end{fact}

\begin{proof}
This follows from the fact that the proportional rank of a $(k^2+tk(k-1),k,1)$-RBIBD is $$\frac{k}{k^2+tk(k-1)}=\frac{1}{(t+1)k-t}$$ and a $(k^2+tk(k-1),k,1)$-RBIBD has $$\frac{\binom{k^2+tk(k-1)}{2}/\binom{k}{2}}{(k^2+tk(k-1))/k}=\frac{k^2+tk(k-1)-1}{k-1}=(t+1)k+1$$ parallel classes. Taking a uniform blow-up gives the desired conclusion.
\end{proof}

The point of Fact \ref{RBIBD} is that, for instance when $r=23$, Theorem \ref{furedi} implies that $\frac{n}{22-1/22}\leq \mc_{23}(K_n)\leq \frac{n}{19}$ (if no affine plane of order 20, 21, or 22 exists).  But by Fact \ref{RBIBD}, if a $(231,11,1)$-RBIBD exists ($k=11$, $t=1$), then $\frac{n}{22-1/22}\leq \mc_{23}(K_n)\leq \frac{n}{21}$.

Also note that for $r=7$, Theorem \ref{furedi} implies that $\frac{6n}{35}\leq \mc_7(K_n)\leq \frac{7n}{35}=\frac{n}{5}$.  It is well known that a $(15, 3, 1)$-RBIBD exists; this is the original Kirkman schoolgirls problem (in fact, four out of the 80 Steiner triple systems on 15 vertices are resolvable -- see \cite{MPR}).  So Fact \ref{RBIBD} implies that there are at least four other examples which show that $\mc_7(K_n)\leq \frac{n}{5}$.

%
%\begin{center}
% \begin{tabular}{|l|c| c|} 
% \hline
%  & proportional rank  $\frac{r(H)}{v(H)}$ &  \# parallel classes (colors) \\
% \hline\hline
%$(k^2,k,1)$-RBIBD  & $k/k^2=1/k$ &  $\frac{\binom{k^2}{2}/\binom{k}{2}}{k^2/k}=k+1$ \\ 
% \hline
%$(k^2+k(k-1),k,1)$-RBIBD  &$k/(k^2+k(k-1))=1/(2k-1)$  & $\frac{\binom{k^2+k(k-1)}{2}/\binom{k}{2}}{(k^2+k(k-1))/k}=2k+1$ \\
% \hline
% $(k^2+2k(k-1),k,1)$-RBIBD  &$k/(k^2+2k(k-1))=1/(3k-2)$  & $\frac{\binom{k^2+2k(k-1)}{2}/\binom{k}{2}}{(k^2+2k(k-1))/k}=3k+1$ \\
% \hline
%  $(k^2+tk(k-1),k,1)$-RBIBD  &$k/(k^2+tk(k-1))=1/((t+1)k-t)$  & $\frac{\binom{k^2+tk(k-1)}{2}/\binom{k}{2}}{(k^2+tk(k-1))/k}=(t+1)k+1$ \\
% \hline
%\end{tabular}
%\end{center}

%Here is a more general problem (stated only for 7 colors though).  
%
%\begin{problem}
%Let $\mathcal{H}_2(7)$ be the set of all hypergraphs $H$ having the property that every pair of vertices is contained in at least one edge and the edge chromatic number of $H$ is at most 7 (that is $H$ can be decomposed into at most $7$ matchings).  What is the minimum value, call it $m_7$, of $\frac{r(H)}{v(H)}$ over all $H\in \mathcal{H}_2(7)$?  It is known that $\frac{6}{35}\leq m_7\leq \frac{7}{35}=\frac{5}{25}=\frac{1}{5}$ where the upper bound comes from an affine plane of order 5 (and a resolvable STS(15), see below) and the lower bound is implied by a result of Furedi.
%\end{problem}

\section{Conclusion}

The main open problem is to prove an analogue of Theorem \ref{GuggiariScott} for $r\geq 4$ colors (the lack of additional evidence prevents us from calling it a conjecture).  Note that the following is true for $r=2,3$.

\begin{problem}
Let $n$ and $r\geq 2$ be positive integers. Prove that if $G$ is a graph on $n$ vertices with $\delta(G)\geq (1-\frac{r-2}{r^2-r})n-1$ and an affine plane of order $r$ exists, then $\mc_r(G)\geq \frac{n}{r-1}$.
\end{problem}

When $r=7$, Theorem \ref{main_example} says that when $42|n$, there exists a graph $G$ on $n$ vertices with $\delta(G)= (1-\frac{5}{42})n-2$, such that $\mc_7(G)<\frac{n}{6}$.  However, Theorem \ref{furedi} says $\mc_7(K_n)\geq \frac{6n}{35}>\frac{n}{6}$ (and it is even conceivable that $\mc_7(K_n)=\frac{n}{5}$).

So we can ask a modified version of the original question of Gy\'arf\'as and S\'ark\"ozy which is different whenever an affine plane of order $r-1$ does not exist.

\begin{problem}
If $G$ is a graph on $n$ vertices, how large does the minimum degree of $G$ need to be so that $\mc_r(G)=\mc_r(K_n)$?
\end{problem}

\end{document}